\documentclass[11pt]{article}
\usepackage[margin=2.4cm]{geometry}
\usepackage{amsmath,amssymb,amsthm,array,booktabs}
\newtheorem{theorem}{Theorem}[section]
\newtheorem{lemma}[theorem]{Lemma}
\newtheorem{proposition}[theorem]{Proposition}
\newtheorem{corollary}[theorem]{Corollary}
\newtheorem{result}[theorem]{Result}
\theoremstyle{definition}
\newtheorem{definition}[theorem]{Definition}
\newtheorem{remark}[theorem]{Remark}
\newcommand{\Z}{\mathbb{Z}}
\newcommand{\Q}{\mathbb{Q}}
\newcommand{\ze}[1]{\zeta_{#1}}
\newcommand{\wh}[1]{\widehat{#1}}
\title{A Family of Gauss Type Hadamard Difference Sets}
\author{Bernhard Schmidt\\
School of Physical and Mathematical Sciences\\	Nanyang Technological University\\
Singapore 637371, Republic of Singapore\\
Email: bernhard@ntu.edu.sg}
\date{\today}

\begin{document}
\maketitle

\begin{abstract}
A  Hadamard difference set (HDS) $D$ of order $u^2$ in an abelian group $G$ satisfies
$|\chi(D)|=u$ for every nontrivial character $\chi$ of $G$.
We call such a character value \emph{naive} if it is divisible by $u$, i.e.,
if it is equal to $u$ times a root of unity.
All previously known abelian HDSs only have naive character values.
We show that for $d\ge1$ and $u=3\cdot2^d$, a group $\Z_3^2\times H$, with $H$ an abelian
group of order $2^{2d+2}$, contains a HDS of order $u^2$ with non-naive character values
if and only if $8\le\exp H\le2^{d+2}$.
All difference sets obtained are new.
The proof rests on a specific HDS in $\Z_3^2\times \Z_8\times \Z_2$,
a covering extended building set on $\Z_3^2\times \Z_8\times \Z_4$, and a variation of
the Davis--Jedwab recursive construction.
\end{abstract}

%%%%%%%%%%%%%%%%%%%%%%%%%%%%%%%%%%%%%%%%%%%%%%%%%%%%%

\section{Introduction}\label{sec:intro}

Let $\alpha$ be a cyclotomic integer with $|\alpha|^2=n^2$ where $n$ is a positive integer.
As in \cite{fe}, we call $\alpha$ \emph{naive} if $\alpha=n \zeta$ where $\zeta$ is a root of unity.
By Kronecker's theorem \cite{kronecker}, a solution $\alpha$ of $|\alpha|^2=n^2$ is naive if and
only if $\alpha$ is divisible by $n$.

\medskip

A \emph{Hadamard difference set (HDS)} of order $u^2$ has parameters
$(v,k,\lambda)=(4u^2,\,2u^2-u,\,u^2-u)$.
Throughout this paper $G$ is a finite abelian group (written
additively) and $\wh G$ its group of complex characters. The trivial character is denoted by $1$.
For $B\subseteq G$ and $\chi \in \wh G$, write $\chi(B)=\sum_{g\in B} \chi(g)$;
the complex numbers $\chi(B)$ with $\chi\ne1$ are called the
\emph{values} of $B$, and the values of a collection of subsets of $G$
are the values of its members.
If $|G|=4u^2$, a subset $D$ of $G$ with $|D|=2u^2-u$ is a HDS in $G$ if and only if
\begin{equation}\label{eq:modu}
|\chi(D)|^2=u^2\qquad\text{for every }\chi\in\wh G\setminus \{1\}.
\end{equation}
Note that $\chi(D)$ is naive if and only if it is divisible by $u$.
We call a HDS $D$ \emph{naive} if $\chi(D)$ is naive for every character $\chi$.
All character values of previously known abelian HDSs
\cite{menon,turyn,mcfarland,turyn84,dillon,davis91,jedwab92,xia,adjs,kraemer,xiangchen,vet,wx,chen,dj,mo}
are naive.
This is no coincidence: the known recursive constructions assemble HDSs from building
blocks whose character values are an integer times a root of unity, and this form is
inherited at every step of the recursion.
Naivety is also what many nonexistence arguments exploit: once $\chi(D)/u$ is known to
be a root of unity, for instance under self-conjugacy assumptions, results such as Turyn's exponent
bound \cite{turyn} and Ma's lemma \cite{ma} heavily restrict the structure of $D$.
Therefore, the existence of non-naive HDSs also is relevant to the HDS nonexistence theory.
We show that such HDSs indeed exist:

\begin{theorem}\label{thm:main}
Let $d\ge1$, $u=3\cdot2^d$, and let $H$ be an abelian $2$-group of
order $2^{2d+2}$. Then $\Z_3^2\times H$ contains a non-naive HDS of order $u^2$
if and only if
\begin{equation} \label{thm:main1}
8\ \le\ \exp H\ \le\ 2^{d+2}.
\end{equation}
\end{theorem}

The necessity of condition \eqref{thm:main1} follows directly from results of Turyn:
Let $D$ be a HDS in $G=\Z_3^2\times H$.
If $\exp H$ divides $4$, then $\exp G$ divides $12$, hence $u$ is self-conjugate modulo
$\exp G$, and every nontrivial character value of $D$ is divisible by $u$, i.e., naive,
by \cite{turyn} (also see \cite[Ch.~6, Lem.~13.2]{be}).
On the other hand, if $\exp H > 2^{d+2}$, then $2$ is self-conjugate modulo $\exp G$
and there is no HDS in $G$ at all, again by \cite{turyn} (also see
\cite[Ch.~6, Thm.~15.11]{be}).
Thus, in the rest of the paper, we are concerned only with the sufficiency of \eqref{thm:main1}.

\medskip

Write $\ze m=e^{2\pi i/m}$ and  $X=1+\ze8+\ze8^3$.
For a positive integer $n$ divisible by $3$, a \emph{Gauss value of
modulus $n$} is a complex number of the form
\[
\frac n3(\ze3-\ze3^2)\,Y\zeta,\qquad Y\in\{X,\overline X\},\quad
\zeta\ \text{a root of unity};
\]
its absolute value is $n$, since
$(\ze3-\ze3^2)\overline{(\ze3-\ze3^2)}=Y\overline Y=3$.
We call these numbers Gauss values because $\ze3-\ze3^2=\sqrt{-3}$ is a quadratic Gauss sum.
All non-naive character values we use in this paper are Gauss values.
We say a HDS of order $u^2$ is of \emph{Gauss type} if at least one of its character
values is a Gauss value of modulus $u$.

\begin{lemma}\label{lem:notroot}
Gauss values are not naive.
\end{lemma}

\begin{proof}
Let $\gamma = \frac n3(\ze3-\ze3^2)\,Y\zeta$ be a Gauss value of modulus $n$.
Since $(\ze3-\ze3^2)\overline{(\ze3-\ze3^2)}=3$ and $Y\overline Y=X\overline X=3$, we
get $|\gamma|^2=n^2$.
Suppose $\gamma/n$ is a root of unity, say $\xi$. Then
$(\ze3-\ze3^2)\,Y\,\zeta=3\xi$, so
$Y=3\xi\zeta^{-1}/(\ze3-\ze3^2) =-\,\xi\zeta^{-1}(\ze3-\ze3^2)$, using $(\ze3-\ze3^2)^2=-3$.
Hence $\xi\zeta^{-1}=-Y/(\ze3-\ze3^2)=Y(\ze3-\ze3^2)/3$ is a root of unity in $\Q(\ze{24})$,
and thus $\xi\zeta^{-1}= \ze8^i\ze3^j$ for some integers $i,j$.
We conclude that
\[
Y\ze8^{-i}=-\ze3^j(\ze3-\ze3^2)\in \Q(\ze8)\cap\Q(\ze3)=\Q,
\]
which contradicts $|Y\ze8^{-i}|^2=|Y|^2=3$.
\end{proof}

To construct Gauss type HDSs in all groups satisfying \eqref{thm:main1}, we proceed as
follows. Section~\ref{sec:prelim} collects all construction methods for building sets
and covering EBSs that we require, including a variation of the Davis--Jedwab recursive
construction \cite{dj}; every result there is stated in a self-contained form,
independent of the recursion set up later.
Section~\ref{sec:supply} supplies, based on a result of Davis and
Jedwab, relative building sets with naive character values on all groups
required later.
Section~\ref{sec:base} provides the two seeds carrying non-naive character
values --- the minimal non-naive HDS $M_1$ in $\Z_3^2\times\Z_8\times \Z_2$
and a covering extended building set $E_3$ on
$\Z_3^2\times\Z_8\times\Z_4$ --- and settles the cases $d\le2$ of
Theorem~\ref{thm:main}. Section~\ref{sec:closure} combines the supply
with the seeds in an induction on $d$ and proves
Theorem~\ref{thm:main} for all $d\ge3$.
The only external input required for our construction is a relative
difference set existence result of Davis and Jedwab, quoted as
Result~\ref{res:djbs}.
We remark that the seeds are out of reach of the recursive machinery
of \cite{dj}, where every HDS is obtained by assembling a covering
EBS with $h\ge4$ blocks on a subgroup of index $h$
\cite[Thm.~2.4]{dj}; since $h$ is a power of $2$ in all Hadamard
cases, merging blocks via \cite[Lem.~2.3]{dj} shows that one may
always take $h=4$ there. However, a four-block covering EBS yielding a
HDS of order $u^2$, $u\in\{6,12\}$, would live on a group $T$ of
order $u^2$ with $3$-part $\Z_3^2$, and it cannot carry a Gauss
value: if $\exp T$ divides $12$, then all its values are naive by the
self-conjugacy argument used above for the necessity of
\eqref{thm:main1}, and if the $2$-part of $T$ has exponent at least
$8$, then $T$ is contained with index $4$ in some $\Z_3^2\times H$
with $\exp H>2^{d+2}$, in which the assembly would produce a HDS,
whereas no HDS exists in such a group, as noted above. Consequently,
neither $M_1$ nor $E_3$ can be obtained from a four-block covering EBS
($M_1$ by assembly, $E_3$ by merging blocks) --- in contrast to every
HDS constructed in \cite{dj}. For $d\ge3$ this obstruction
disappears, and four-block covering EBSs with Gauss values are exactly
what our recursion produces.

%%%%%%%%%%%%%%%%%%%%%%%%%%%%%%%%%%%%%%%%%%%%%%%%%%%%%%%%%%%%%%%%%%%%%%%%%%%

\section{Building Set Framework}\label{sec:prelim}

Throughout the paper, $W=\Z_3^2$.

\medskip

All constructions in this section are due to Davis and Jedwab \cite{dj},
up to minor variations. However, while \cite{dj} only records the absolute
values of the relevant character sums, we need to keep track of the
character values themselves. We therefore restate the results in the form
needed later, with the values recorded, and include the short proofs; a
precise reference to \cite{dj} accompanies each statement.

The concepts (i)--(iii) below were introduced by Davis and Jedwab
\cite[\S2]{dj}, and (iv) only makes their term ``covering'' pointwise;
(v) is specific to the present paper.

\begin{definition}\label{def:bs}
Let $a,m,s\ge 1$, let $K$ be an abelian group, and let $U\le K$.
\begin{enumerate}
\item[(i)]
A \emph{building block} in $K$
with modulus $m$ is a subset $B$ of $K$ such that
$|\chi(B)|\in \{0,m\}$ for all nontrivial characters $\chi$ of $K$.
\item[(ii)]
An  \emph{$(a,m,s)$ building set relative to $U$}
is a collection $\{B_1,\dots,B_s\}$ of building blocks,
each of cardinality $a$ and with modulus $m$,
 such that the following holds for every nontrivial $\chi\in\wh K$: if $\chi$ is
trivial on $U$ then $\chi(B_i)=0$ for all $i$; otherwise
$|\chi(B_i)|=m$ for exactly one $i$.
\item[(iii)]
An \emph{$(a,m,s,-)$ covering extended building set (covering
EBS) on $K$} is a collection $\{B_1,\dots,B_s\}$ of building blocks with
modulus $m$, of sizes $|B_1|=a-m$ and $|B_i|=a$ $(i\ge2)$, such that every
nontrivial $\chi\in\wh K$ satisfies $|\chi(B_i)|=m$ for exactly one $i$.
The sign ``$-$'' in the notation, taken from \cite{dj}, records that the
exceptional block is deficient, of size $a-m$ rather than $a+m$.
\item[(iv)]
We say that a character $\chi$ is \emph{covered} by a  building block $B_i$ with modulus $m$
if $|\chi(B_i)|=m$.
\item[(v)]
We say that $\chi$ is \emph{Gauss-covered} by a building block $B_i$
with modulus $m$ if $\chi$ is covered by $B_i$ and $\chi(B_i)$ is a
Gauss value of modulus $m$, and that $\chi$ is \emph{Gauss-covered} by
a building set or covering EBS $\{B_1,\dots,B_s\}$ if it is
Gauss-covered by the block covering it (by (ii) and (iii), a character
covered at all is covered by exactly one block).
\end{enumerate}
\end{definition}

The next lemma, which is the two-block case of
\cite[Theorem~2.4]{dj} with character values tracked, converts a
two-block covering EBS into a HDS on a group twice as large.

\begin{lemma}\label{lem:assembly2}
Let $\{B_1,B_2\}$ be an $(m^2,m,2,-)$ covering EBS on an abelian
group $K$ of order $2m^2$, let $G\supseteq K$ be abelian with
$|G|=4m^2$, fix $g\in G\setminus K$, and put $D=B_1\cup(g+B_2)$.
Then $D$ is a HDS in $G$ of order $m^2$. If $\chi\in\wh G$ is
nontrivial on $K$ and $i$ is the index with $|\chi(B_i)|=m$, then
$\chi(D)=\chi(g)^{i-1}\chi(B_i)$; if $\chi\ne1$ is trivial on
$K$, then $\chi(D)=-m$.
\end{lemma}

\begin{proof}
The sets $B_1$ and $g+B_2$ lie in distinct cosets of $K$, so
$\chi(D)=\chi(B_1)+\chi(g)\chi(B_2)$ for every $\chi\in\wh G$.
If $\chi$ is nontrivial on $K$, exactly one summand is nonzero, which
gives the stated value of absolute value $m$. If $\chi\ne1$ is trivial
on $K$, then $\chi(g)=-1$ and
$\chi(D)=|B_1|-|B_2|=(m^2-m)-m^2=-m$.
Thus $|\chi(D)|=m$ for every nontrivial $\chi\in\wh G$, and
$|D|=2m^2-m$, so $D$ is a HDS in $G$ of order $m^2$ by \eqref{eq:modu}.
\end{proof}

The four-block counterpart of Lemma~\ref{lem:assembly2} can be carried
out in two steps: the following lemma, which is the index~$2$ case of
\cite[Lemma~2.3]{dj} with character values tracked, performs the first
step, and
chaining it with Lemma~\ref{lem:assembly2} yields a HDS from a
four-block covering EBS on $K$ in any abelian group containing $K$ as
a subgroup of index $4$ (see the proof of Theorem~\ref{thm:main}).

\begin{lemma}\label{lem:partial}
Let $\{E_1,\dots,E_4\}$ be a $(2m^2,2m,4,-)$ covering EBS on an
abelian group $K$ of order $4m^2$, let $L\supseteq K$ be abelian with
$[L:K]=2$, fix $r\in L\setminus K$, and let $\{i,j,k\}=\{2,3,4\}$.
Then
\[
C_1=E_1\ \cup\ (r+E_i),\qquad C_2=E_j\ \cup\ (r+E_k)
\]
form a $((2m)^2,2m,2,-)$ covering EBS on $L$. Every nonzero value of
$C_1,C_2$ is a root of unity times a value of $\{E_1,\dots,E_4\}$;
in particular, if some character of $K$ is Gauss-covered by
$\{E_1,\dots,E_4\}$, then some character of $L$ is Gauss-covered by
$\{C_1,C_2\}$.
\end{lemma}

\begin{proof}
The sizes are $|C_1|=(2m^2-2m)+2m^2=(2m)^2-2m$ and $|C_2|=(2m)^2$, as required.
Let $\lambda\in\wh L$ be nontrivial and $\lambda'=\lambda|_K$, so
$\lambda(C_1)=\lambda'(E_1)+\lambda(r)\lambda'(E_i)$ and
$\lambda(C_2)=\lambda'(E_j)+\lambda(r)\lambda'(E_k)$.
If $\lambda'\ne1$, then exactly one of the four values
$\lambda'(E_1),\dots,\lambda'(E_4)$ is nonzero, of absolute value
$2m$; hence exactly one of $\lambda(C_1),\lambda(C_2)$ is nonzero, and
it is a root of unity times a value of $\{E_1,\dots,E_4\}$, of
absolute value $2m$. If $\lambda'=1$, then $\lambda(r)=-1$ and
$\lambda(C_1)=|E_1|-|E_i|=-2m$, $\lambda(C_2)=|E_j|-|E_k|=0$.
Thus every nontrivial $\lambda\in\wh L$ is covered by exactly one
block, as required. For the final statement, note that both extensions
$\lambda$ of a character $\lambda'\ne1$ of $K$ satisfy
$|\lambda(C_1)|+|\lambda(C_2)|=2m$, and that a root of unity times a
Gauss value of modulus $2m$ is a Gauss value of modulus $2m$.
\end{proof}

The following lemma assembles a two-block covering EBS and two
one-block relative building sets, living on the three quotients of a
group $K$ by the involutions of a subgroup $Q\cong\Z_2^2$, into a
four-block covering EBS on $K$; it is a variation of the
Davis--Jedwab constructions \cite[Theorems~3.2 and 4.3]{dj}.

\begin{lemma}\label{lem:main}
Let $m\ge1$, let $K$ be abelian of order $4m^2$, $Q\le K$ with
$Q\cong\Z_2^2$, and
let $U,H_1,H_2$ be the three subgroups of order $2$ of $Q$, with
canonical epimorphisms $\pi_U,\pi_1,\pi_2$ onto $K/U$, $K/H_1$,
$K/H_2$. Suppose that
\begin{enumerate}
\item[(i)] $\{B_1,B_2\}$ is an $(m^2,m,2,-)$ covering EBS on $K/U$,
\item[(ii)] for $i=1,2$, $R_i$ is a building set with one block of
size $m^2$ and modulus $m$ on $K/H_i$ relative to $Q/H_i$.
\end{enumerate}
Then
\[
\mathcal E\;=\;\bigl\{\pi_U^{-1}(B_1),\ \pi_U^{-1}(B_2),\
\pi_1^{-1}(R_1),\ \pi_2^{-1}(R_2)\bigr\}
\]
is a $(2m^2,2m,4,-)$
covering EBS on $K$ whose values are twice the values of the inputs;
in particular, each character of $K/U$ that is Gauss-covered by
$\{B_1,B_2\}$ lifts to a character of $K$ that is Gauss-covered by
$\mathcal E$.
\end{lemma}

\begin{proof}
The sizes of the blocks in $\mathcal E$ are $2(m^2-m)=2m^2-2m$ and $2m^2,2m^2,2m^2$, as
required for a $(2m^2,2m,4,-)$
covering EBS.

Let $\chi\in\wh K$ be nontrivial. Note that the  kernel of $\chi|_Q$ contains at least one
of $U,H_1,H_2$; if $\chi$ is nontrivial on $Q$, the kernel is exactly
one of the three. A block pulled back through a kernel of order $2$
satisfies $\chi(\pi^{-1}(B))=2\widetilde\chi(B)$ if $\chi$ is trivial
on the kernel (with $\widetilde\chi$ the induced character) and $=0$
otherwise. We check the three cases.

If $\ker(\chi|_Q)\supseteq U$ (including $\chi$ trivial on $Q$): the
blocks $\pi_i^{-1}(R_i)$ contribute $0$ --- either $\chi$ is nontrivial
on $H_i$, or $\chi$ is trivial on all of $Q$ and then the induced
character on $K/H_i$ is principal on $Q/H_i$ and nontrivial, so
$R_i$ gives $0$ by relativity. Write $\chi=\lambda\circ\pi_U$ with
$\lambda\in\wh{K/U}$ nontrivial; exactly one $i$ has
$|\lambda(B_i)|=m$, so exactly one block of $\mathcal E$ covers $\chi$,
with value $2\lambda(B_i)$ of modulus $2m$.

If $\ker(\chi|_Q)=H_1$ exactly: $\chi$ is nontrivial on $U$ and on
$H_2$, so the two lifted blocks and $\pi_2^{-1}(R_2)$ give $0$. The
induced character $\mu$ on $K/H_1$ is nonprincipal on $Q/H_1$, so
$|\mu(R_1)|=m$ and $\chi(\pi_1^{-1}(R_1))=2\mu(R_1)$ has modulus $2m$.
The case $\ker(\chi|_Q)=H_2$ is symmetric.

Thus every nontrivial character is covered exactly once and with modulus $2m$.
\end{proof}

\section{Supply of Relative Building Sets}\label{sec:supply}

The induction of Section~\ref{sec:closure} consumes, at every step,
relative building sets of one normalized shape, which we now supply. This section rests on a result of Davis
and Jedwab, quoted as Result~\ref{res:djbs} below, which is the only
external input to our construction. All character values occurring in
this section are naive.

Every group occurring from now on is of the form $W\times A$ with $A$
an abelian $2$-group, which we call the \emph{$2$-part}; every
subgroup of order $2$ of $W\times A$ is contained in $A$, as $|W|$ is
odd.

\begin{definition}\label{def:rfsets}
For $d\ge1$, write $m_d=3\cdot2^{d-1}$; thus the parameter
$u=3\cdot2^d$ of Theorem~\ref{thm:main} equals $2m_d$.
Let $d\ge2$ and $m=m_d$, let $T$ be an abelian group of order
$2m^2$, and let $V\le T$ be a subgroup of order $2$. An \emph{$R$-set
of modulus $m$ on $(T,V)$} is a building set with one block of size
$m^2$ and modulus $m$ on $T$ relative to $V$.
\end{definition}

\begin{result}[{Davis--Jedwab \cite[Cor.~8.1(i)]{dj}}]\label{res:djbs}
For each $c\ge1$ and each abelian group $S$ of order $2^{2c+1}$ with
$\exp S\le2^{c+1}$, there exists a
$(2^{2c}\cdot9,\,2,\,2^{2c}\cdot9,\,2^{2c-1}\cdot9)$ semiregular
relative difference set in $W\times S$ relative to any subgroup of
order $2$.
\end{result}

In our terms, Result~\ref{res:djbs} reads as follows.

\begin{theorem}[Supply]\label{thm:supply}
Let $d\ge2$, let $T$ be an abelian group of order $2m_d^2$
with $3$-part $W$ whose $2$-part has exponent at most $2^d$, and
let $V\le T$ be any subgroup of order $2$. Then there is an $R$-set
of modulus $m_d$ on $(T,V)$.
\end{theorem}

\begin{proof}
The $2$-part of $T$ has order $2^{2d-1}$ and exponent at most
$2^d$, so Result~\ref{res:djbs}, applied with
$c=d-1$, provides a
$(m_d^2,\,2,\,m_d^2,\,m_d^2/2)$ semiregular relative
difference set $R$ in $T$ relative to $V$. By the character criterion
for relative difference sets \cite[Lemma~1.1(ii)]{dj}, this means
precisely that $|R|=m_d^2$, that $\chi(R)=0$ for every nontrivial
character $\chi$ of $T$ that is principal on $V$, and that
$|\chi(R)|=m_d$ for every character of $T$ that is nonprincipal on
$V$. Thus $R$ is an $R$-set of modulus $m_d$ on $(T,V)$.
\end{proof}

%%%%%%%%%%%%%%%%%%%%%%%%%%%%%%%%%%%%%%%%%%%%%%%%%%%%%

\section{Seeds}\label{sec:base}

This section provides the seeds of our construction, which carry
the non-naive character values:

\begin{center}
\renewcommand{\arraystretch}{1.15}
\begin{tabular}{@{}lll@{}}
\toprule
Seed & Group & Constructed in \\
\midrule
HDS $M_1$ of order $36$  & $W\times\Z_8\times\Z_2$ & Theorem~\ref{thm:M1}\\
covering EBS from doubling $M_1$ & $W\times\Z_8\times\Z_2^2$ & Corollary~\ref{cor:M1double}\\
covering EBS $E_3$       & $W\times\Z_8\times\Z_4$ & Theorem~\ref{thm:seed}\\
\bottomrule
\end{tabular}
\end{center}

\noindent
The HDS of order $36$ enters the constructions through its doubling
(Corollary~\ref{cor:M1double}), which provides a second covering EBS with
Gauss values, on $W\times\Z_8\times\Z_2^2$: it is used in
Proposition~\ref{prop:smalld} below and in the base case of
Theorem~\ref{thm:stock}.

\subsection{The HDS $M_1$ of order $36$}

Elements of
$\Gamma_1=W\times\Z_8\times\Z_2$ are written $(u,v,x,y)$ with
$(u,v)\in W$, $x\in\Z_8$, $y\in\Z_2$ (throughout this section, the
coordinate pair $(u,v)$ ranges over $W$ and is unrelated to the order
parameter $u$), and the characters of $\Gamma_1$
are $\psi_{(p,q)}\chi_{(s,t)}$ with
$\psi_{(p,q)}(u,v)=\ze3^{pu+qv}$ and
$\chi_{(s,t)}(x,y)=\ze8^{sx}(-1)^{ty}$.

\begin{theorem}\label{thm:M1}
The set $M_1\subseteq W\times\Z_8\times\Z_2$ displayed below is a HDS
with parameters $(144,66,30)$ and order $36$, and the character
$\psi_{(0,1)}\chi_{(1,0)}$ has the Gauss value
\begin{equation}\label{eq:cert144}
\psi_{(0,1)}\chi_{(1,0)}(M_1)=2(\ze3-\ze3^2)\,\overline X\,\ze{24}^{22}.
\end{equation}
\end{theorem}

Set $M_1=\bigcup_{(u,v)\in W}\{(u,v)\}\times M_{(u,v)}$, where
$M_{(u,v)}\subseteq\Z_8\times\Z_2$ are given by the following $0/1$
matrices (row index $x=0,\dots,7$ from top, column index $y=0,1$ from
left):

\bigskip
\begingroup\scriptsize\setlength{\tabcolsep}{2.6pt}
\noindent
$M_{(0,0)}=$\;\begin{tabular}{|cc|}\hline
1&1\\1&1\\1&1\\0&0\\0&0\\0&0\\0&0\\1&1\\\hline
\end{tabular}\hspace{8pt}
$M_{(0,1)}=$\;\begin{tabular}{|cc|}\hline
1&1\\0&0\\0&0\\1&1\\0&0\\1&1\\1&1\\0&0\\\hline
\end{tabular}\hspace{8pt}
$M_{(0,2)}=$\;\begin{tabular}{|cc|}\hline
1&1\\1&0\\1&0\\1&1\\0&0\\0&1\\1&0\\1&1\\\hline
\end{tabular}\hspace{8pt}
$M_{(1,0)}=$\;\begin{tabular}{|cc|}\hline
1&0\\1&0\\1&0\\1&0\\0&1\\0&1\\0&1\\0&1\\\hline
\end{tabular}\hspace{8pt}
$M_{(1,1)}=$\;\begin{tabular}{|cc|}\hline
1&0\\1&1\\0&1\\0&0\\1&0\\1&1\\0&1\\0&0\\\hline
\end{tabular}

\medskip
\noindent
$M_{(1,2)}=$\;\begin{tabular}{|cc|}\hline
0&1\\0&0\\0&1\\0&0\\0&1\\0&0\\0&1\\0&0\\\hline
\end{tabular}\hspace{8pt}
$M_{(2,0)}=$\;\begin{tabular}{|cc|}\hline
1&0\\1&1\\0&1\\0&0\\1&0\\1&1\\0&1\\0&0\\\hline
\end{tabular}\hspace{8pt}
$M_{(2,1)}=$\;\begin{tabular}{|cc|}\hline
0&1\\1&0\\0&1\\0&1\\1&0\\0&1\\1&0\\1&0\\\hline
\end{tabular}\hspace{8pt}
$M_{(2,2)}=$\;\begin{tabular}{|cc|}\hline
0&1\\0&0\\0&1\\0&0\\0&1\\0&0\\0&1\\0&0\\\hline
\end{tabular}
\endgroup
\bigskip

\noindent
We have $|M_{(0,2)}|=10$, $|M_{(1,2)}|=|M_{(2,2)}|=4$ and
$|M_{(u,v)}|=8$ otherwise (total $66$).
Theorem~\ref{thm:M1} was verified by exact integer
computation.

\medskip

In addition to $M_1$ itself, the constructions use the following
two-block covering EBS obtained by doubling $M_1$; a block of the form
$B$ below already appears in \cite[Lemma~6.3]{dj}.

\begin{corollary}\label{cor:M1double}
On $K_1=\Gamma_1\times\Z_2$, put
\[
\widetilde B=M_1\times\Z_2,\qquad
B=M_1\times\{0\}\ \cup\ (\Gamma_1\setminus M_1)\times\{1\}.
\]
Then $\{\widetilde B,B\}$ is a $(144,12,2,-)$ covering EBS on
$K_1\cong W\times\Z_8\times\Z_2^2$ which Gauss-covers a
character.
\end{corollary}

\begin{proof}
Write $\kappa$ for the character of $K_1$ with kernel $\Gamma_1$.
Clearly $(\chi\times1)(\widetilde B)=2\chi(M_1)$ and
$(\chi\times\kappa)(\widetilde B)=0$ for all $\chi\in\wh{\Gamma_1}$.
For $\chi\ne1$ we have $\chi(\Gamma_1)=0$, so
$(\chi\times1)(B)=\chi(M_1)+\chi(\Gamma_1\setminus M_1)=0$ and
$(\chi\times\kappa)(B)=\chi(M_1)-\chi(\Gamma_1\setminus M_1)
=2\chi(M_1)$, while
$(1\times\kappa)(B)=|M_1|-|\Gamma_1\setminus M_1|=66-78=-12$.
Since $|\chi(M_1)|=6$ for all $\chi\ne1$ by Theorem~\ref{thm:M1} and
\eqref{eq:modu}, every nontrivial character of $K_1$ is covered by
exactly one of the two blocks, with a value of absolute value $12$.
As $|\widetilde B|=132=144-12$ and $|B|=144$, the pair
$\{\widetilde B,B\}$ is a $(144,12,2,-)$ covering EBS on $K_1$.
Finally, by \eqref{eq:cert144}, the value of $\widetilde B$ at
$(\psi_{(0,1)}\chi_{(1,0)})\times1$ is twice a Gauss value of
modulus $6$, hence a Gauss value of modulus $12$.
\end{proof}

\begin{remark}\label{rem:rdslift}
The block $B$ is the classical lift of a HDS to a relative difference
set: a subset $D$ of an abelian group $G$ of order $4u^2$ is a HDS of
order $u^2$ if and only if $D\times\{0\}\cup(G\setminus
D)\times\{1\}$ is a semiregular $(4u^2,2,4u^2,2u^2)$ relative
difference set in $G\times\Z_2$ relative to $\{0\}\times\Z_2$. In
particular, the values $2\chi(M_1)$ show that semiregular relative
difference sets with non-naive character values exist as well.
\end{remark}

\subsection{The covering EBS $E_3$}

Elements of $\Gamma_2=W\times\Z_8\times\Z_4$ are written $(u,v,x,z)$
with $(u,v)\in W$, $x\in\Z_8$, $z\in\Z_4$, and we recall
$X=1+\ze8+\ze8^3$, $X\overline X=3$. Set
\[
C_n=\bigcup_{(u,v)\in W} \{(u,v)\}\times C^{(n)}_{(u,v)}
\qquad(n=1,2),
\]
where $C^{(n)}_{(u,v)}\subseteq\Z_8\times\Z_4$ are given by the
following $0/1$ matrices (row index $x=0,\dots,7$ from top, column
index $z=0,\dots,3$ from left):

\bigskip
\begingroup\scriptsize\setlength{\tabcolsep}{2.6pt}
\noindent
$C^{(1)}_{(0,0)}=$\;\begin{tabular}{|cccc|}\hline
1&0&1&1\\
1&0&1&1\\
1&0&1&1\\
0&1&0&0\\
1&0&1&1\\
1&0&1&1\\
1&0&1&1\\
0&1&0&0\\\hline
\end{tabular}\hspace{8pt}
$C^{(1)}_{(0,1)}=$\;\begin{tabular}{|cccc|}\hline
0&0&0&0\\
0&0&0&0\\
0&0&0&0\\
1&1&1&1\\
0&0&0&0\\
0&0&0&0\\
0&0&0&0\\
1&1&1&1\\\hline
\end{tabular}\hspace{8pt}
$C^{(1)}_{(0,2)}=$\;\begin{tabular}{|cccc|}\hline
0&0&0&0\\
0&0&0&0\\
0&0&0&0\\
1&1&1&1\\
0&0&0&0\\
0&0&0&0\\
0&0&0&0\\
1&1&1&1\\\hline
\end{tabular}\hspace{8pt}
$C^{(1)}_{(1,0)}=$\;\begin{tabular}{|cccc|}\hline
0&1&0&1\\
0&1&0&1\\
0&1&0&1\\
1&0&1&0\\
0&1&0&1\\
0&1&0&1\\
0&1&0&1\\
1&0&1&0\\\hline
\end{tabular}\hspace{8pt}
$C^{(1)}_{(1,1)}=$\;\begin{tabular}{|cccc|}\hline
1&1&0&0\\
1&1&0&0\\
1&1&0&0\\
0&0&1&1\\
1&1&0&0\\
1&1&0&0\\
1&1&0&0\\
0&0&1&1\\\hline
\end{tabular}

\medskip
\noindent
$C^{(1)}_{(1,2)}=$\;\begin{tabular}{|cccc|}\hline
0&1&1&0\\
0&1&1&0\\
0&1&1&0\\
1&0&0&1\\
0&1&1&0\\
0&1&1&0\\
0&1&1&0\\
1&0&0&1\\\hline
\end{tabular}\hspace{8pt}
$C^{(1)}_{(2,0)}=$\;\begin{tabular}{|cccc|}\hline
0&1&0&1\\
0&1&0&1\\
0&1&0&1\\
1&0&1&0\\
0&1&0&1\\
0&1&0&1\\
0&1&0&1\\
1&0&1&0\\\hline
\end{tabular}\hspace{8pt}
$C^{(1)}_{(2,1)}=$\;\begin{tabular}{|cccc|}\hline
0&1&1&0\\
0&1&1&0\\
0&1&1&0\\
1&0&0&1\\
0&1&1&0\\
0&1&1&0\\
0&1&1&0\\
1&0&0&1\\\hline
\end{tabular}\hspace{8pt}
$C^{(1)}_{(2,2)}=$\;\begin{tabular}{|cccc|}\hline
1&1&0&0\\
1&1&0&0\\
1&1&0&0\\
0&0&1&1\\
1&1&0&0\\
1&1&0&0\\
1&1&0&0\\
0&0&1&1\\\hline
\end{tabular}

\medskip
\noindent
$C^{(2)}_{(0,0)}=$\;\begin{tabular}{|cccc|}\hline
1&1&0&0\\
0&1&1&0\\
1&1&0&0\\
1&0&0&1\\
0&0&1&1\\
1&0&0&1\\
0&0&1&1\\
0&1&1&0\\\hline
\end{tabular}\hspace{8pt}
$C^{(2)}_{(0,1)}=$\;\begin{tabular}{|cccc|}\hline
0&1&0&1\\
0&1&0&1\\
1&0&1&0\\
0&1&0&1\\
1&0&1&0\\
1&0&1&0\\
0&1&0&1\\
1&0&1&0\\\hline
\end{tabular}\hspace{8pt}
$C^{(2)}_{(0,2)}=$\;\begin{tabular}{|cccc|}\hline
0&0&0&0\\
1&1&1&1\\
0&0&0&0\\
0&0&0&0\\
1&1&1&1\\
0&0&0&0\\
1&1&1&1\\
1&1&1&1\\\hline
\end{tabular}\hspace{8pt}
$C^{(2)}_{(1,0)}=$\;\begin{tabular}{|cccc|}\hline
1&0&1&0\\
1&0&1&0\\
1&0&1&0\\
1&0&1&0\\
0&1&0&1\\
0&1&0&1\\
0&1&0&1\\
0&1&0&1\\\hline
\end{tabular}\hspace{8pt}
$C^{(2)}_{(1,1)}=$\;\begin{tabular}{|cccc|}\hline
1&1&0&0\\
0&1&1&0\\
1&1&0&0\\
1&0&0&1\\
0&0&1&1\\
1&0&0&1\\
0&0&1&1\\
0&1&1&0\\\hline
\end{tabular}

\medskip
\noindent
$C^{(2)}_{(1,2)}=$\;\begin{tabular}{|cccc|}\hline
0&0&0&0\\
0&0&0&0\\
1&1&1&1\\
1&1&1&1\\
1&1&1&1\\
1&1&1&1\\
0&0&0&0\\
0&0&0&0\\\hline
\end{tabular}\hspace{8pt}
$C^{(2)}_{(2,0)}=$\;\begin{tabular}{|cccc|}\hline
0&0&1&1\\
0&1&1&0\\
0&0&1&1\\
1&0&0&1\\
1&1&0&0\\
1&0&0&1\\
1&1&0&0\\
0&1&1&0\\\hline
\end{tabular}\hspace{8pt}
$C^{(2)}_{(2,1)}=$\;\begin{tabular}{|cccc|}\hline
0&0&1&1\\
0&1&1&0\\
0&0&1&1\\
1&0&0&1\\
1&1&0&0\\
1&0&0&1\\
1&1&0&0\\
0&1&1&0\\\hline
\end{tabular}\hspace{8pt}
$C^{(2)}_{(2,2)}=$\;\begin{tabular}{|cccc|}\hline
0&0&0&0\\
1&0&0&1\\
1&0&1&0\\
0&1&1&0\\
1&1&1&1\\
0&1&1&0\\
0&1&0&1\\
1&0&0&1\\\hline
\end{tabular}
\endgroup
\bigskip

\noindent
We have $|C_1|=132$ and $|C_2|=144$.

\begin{theorem}\label{thm:seed}
The pair $E_3=\{C_1,C_2\}$ displayed above is a $(144,12,2,-)$
covering EBS on $\Gamma_2$. The character
$\lambda(u,v,x,z)=\ze3^{\,v}\ze8^{\,x}(-1)^{z}$ is Gauss-covered, with
\[
\lambda(C_1)=0,\qquad \lambda(C_2)=4(\ze3-\ze3^2)\,X\,\ze3^{2}.
\]
\end{theorem}

All assertions of Theorem~\ref{thm:seed} were verified by exact
integer computation. The pair $E_3$ was obtained from a computer-found
HDS of order $144$ in $W\times\Z_8^2$; note that, conversely,
assembling $E_3$ into $W\times\Z_8^2$ by Lemma~\ref{lem:assembly2}
recovers a HDS of order $144$.

%%%%%%%%%%%%%%%%%%%%%%%%%%%%%%%%%%%%%%%%%%%%%%%

\medskip

With the seeds in hand, the cases $d\le2$ of
Theorem~\ref{thm:main} can be settled at once; the machinery of the
following two sections is needed only for $d\ge3$.

\begin{proposition}\label{prop:smalld}
Theorem~\ref{thm:main} holds for $d\le2$.
\end{proposition}

\begin{proof}
For $d=1$, condition \eqref{thm:main1} forces $H\cong\Z_8\times\Z_2$,
so Theorem~\ref{thm:M1} proves the claim. Let $d=2$. By
Corollary~\ref{cor:M1double}, the doubling of $M_1$ is a
$(144,12,2,-)$ covering EBS on $W\times\Z_8\times\Z_2^2$;
together with $E_3$ (Theorem~\ref{thm:seed}), we have
two covering EBSs with Gauss values of modulus $12$, on
$W\times\Z_8\times\Z_2^2$ and on $W\times\Z_8\times\Z_4$. By
Lemma~\ref{lem:assembly2}, each of them
yields a HDS of order $144$ in every abelian group that contains the
respective group as a subgroup of index $2$, and these HDSs are of Gauss type,
as their character values are roots of unity times values of the EBSs.
The groups reached in this way are $W\times H$ with
$H\in\{\Z_{16}\times\Z_2^2,\ \Z_8\times\Z_4\times\Z_2,\
\Z_8\times\Z_2^3\}$ from the first EBS and
$H\in\{\Z_8\times\Z_8,\ \Z_{16}\times\Z_4,\
\Z_8\times\Z_4\times\Z_2\}$ from the second --- together all
five abelian groups $H$ of order $64$ with $8\le\exp H\le16$, as
required by \eqref{thm:main1} for $d=2$.
\end{proof}

%%%%%%%%%%%%%%%%%%%%%%%%%%%%%%%%%%%%%%%%%%%%%%%%%%%%%%%%%%%%%%%%%%%%%%%%%%%%%%%%%%%%%%%%%%%%%%%%%%%%%%%

\section{The Closure}\label{sec:closure}

In view of Proposition~\ref{prop:smalld}, it remains to prove
Theorem~\ref{thm:main} for $d\ge3$. This is done by an induction on
$d$ that constructs two families of covering EBSs with Gauss values
at once: two-block ones on groups of order $2m_d^2$ and four-block
ones on groups of order $4m_d^2$. All groups in this section have
$3$-part $W$.

\begin{theorem}[Stock]\label{thm:stock}
Let $d\ge3$ and let $A$ be an abelian $2$-group with
$8\le\exp A\le2^d$.
\begin{enumerate}
\item[(a)] If $|A|=2^{2d-1}$, then $T=W\times A$ carries an
$(m_d^2,m_d,2,-)$ covering EBS which Gauss-covers a character.
\item[(b)] If $|A|=2^{2d}$, then $\Gamma=W\times A$ carries a
$(2m_d^2,2m_d,4,-)$ covering EBS which Gauss-covers a character.
\end{enumerate}
\end{theorem}

\begin{proof}
We prove (a) and (b) simultaneously by induction on $d$: part (a)
follows from the seeds when $d=3$ and from part (b) for $d-1$ when
$d\ge4$, and part (b) then follows from part (a) for the same $d$.

\medskip

(a) For $d=3$, note that $A$ has order $32$ and exponent $8$,
so $A\cong\Z_8\times\Z_4$ or $A\cong\Z_8\times\Z_2^2$. In the first
case, $T$ carries the covering EBS $E_3$ of Theorem~\ref{thm:seed};
in the second case, it carries the doubling of $M_1$
(Corollary~\ref{cor:M1double}). In both cases the parameters are
$(144,12,2,-)=(m_3^2,m_3,2,-)$.

Let $d\ge4$ and write $A=\bigoplus_i\langle e_i\rangle$ with
cyclic groups $\langle e_i\rangle$. We claim that $A$ has a subgroup
$A_0$ of index $2$ with $8\le\exp A_0\le2^{d-1}$. Note that
\emph{halving} one summand, that is, passing to
$A_0=\langle2e_j\rangle\oplus\bigoplus_{i\ne j}\langle e_i\rangle$,
yields a subgroup of index $2$, and that $A$ has at least two
summands, since $\exp A\le2^d<|A|$. If $\exp A\le2^{d-1}$, we
halve a summand chosen so that the exponent is preserved: a summand
of maximal order if at least two summands have maximal order, and any
other summand otherwise. If $\exp A=2^d$, then exactly one summand
has order $2^d$, as two such summands would exceed
$|A|=2^{2d-1}$; halving it gives $\exp A_0=2^{d-1}$, since all
other summands have order at most $2^{d-1}$. In both cases
$8\le\exp A_0\le2^{d-1}$, using $d-1\ge3$.

Now $K=W\times A_0$ has order $4m_{d-1}^2$ and carries a
$(2m_{d-1}^2,2m_{d-1},4,-)$ covering EBS which Gauss-covers a
character, by part (b) for $d-1$. Since $[T:K]=2$
and $2m_{d-1}=m_d$, Lemma~\ref{lem:partial} turns it into an
$(m_d^2,m_d,2,-)$ covering EBS on $T$ which Gauss-covers a
character.

\medskip

(b) Write $A=\bigoplus_i\langle e_i\rangle$ with cyclic groups
$\langle e_i\rangle$; there are at least two summands, since
$\exp A\le2^d<|A|$. Choose an index $j$ as follows: if at least
two summands have maximal order, let $\langle e_j\rangle$ be one of
them; otherwise, let $\langle e_j\rangle$ be any summand of
nonmaximal order. Let $v$ be the involution of $\langle e_j\rangle$,
let $w$ be the involution of $\langle e_k\rangle$ for some $k\ne j$,
and put $Q=\langle v,w\rangle\cong\Z_2^2$. Let $U=\langle v\rangle$,
$H_1=\langle w\rangle$ and $H_2=\langle v+w\rangle$ be the three
subgroups of order $2$ of $Q$.

Since $\langle v\rangle\le\langle e_j\rangle$, the $2$-part of
$\Gamma/U$ is obtained from $A$ by replacing the summand
$\langle e_j\rangle$ with a cyclic group of half the order. By the
choice of $j$, its exponent equals $\exp A$, so
part (a) provides an $(m_d^2,m_d,2,-)$ covering EBS
on $\Gamma/U$, a group of order $2m_d^2$, which Gauss-covers a
character. The quotients $\Gamma/H_1$
and $\Gamma/H_2$ also have order $2m_d^2$, and the exponents of
their $2$-parts divide $\exp A\le2^d$; hence
Theorem~\ref{thm:supply} provides $R$-sets of modulus $m_d$ on
$(\Gamma/H_i,\,Q/H_i)$, $i=1,2$. Now Lemma~\ref{lem:main} yields a
$(2m_d^2,2m_d,4,-)$ covering EBS on $\Gamma$, and a Gauss-covered
character of $\Gamma/U$ lifts to a Gauss-covered character of
$\Gamma$.
\end{proof}

\begin{lemma}[Target selection]\label{lem:select}
Let $d\ge3$ and let $H$ be an abelian $2$-group of order
$2^{2d+2}$ with $8\le\exp H\le2^{d+2}$. Then $H$ has a subgroup
$H_0$ of index $4$ with $8\le\exp H_0\le2^d$.
\end{lemma}

\begin{proof}
Write $H=\bigoplus_i\langle e_i\rangle$ with cyclic groups
$\langle e_i\rangle$. If $\exp H\le2^d$, we halve summands twice,
each time preserving the exponent as in the proof of
Theorem~\ref{thm:stock}(a); this is possible, as at each stage the
group has order greater than its exponent and hence at least two
summands. If $\exp H=2^{d+1}$ and two summands have order
$2^{d+1}$, then $H\cong\Z_{2^{d+1}}^2$, as $|H|=2^{2d+2}$;
in this case we halve both summands. If $\exp H=2^{d+1}$ and only
one summand has order $2^{d+1}$, we halve that summand, which
yields exponent $2^d$, and then halve once more, preserving the
exponent. If $\exp H=2^{d+2}$, then only one summand has order
$2^{d+2}$, say $\langle e_j\rangle$, and all other summands have
order at most $2^d$, as $|H|=2^{2d+2}$; in this case we replace
$\langle e_j\rangle$ by $\langle4e_j\rangle$, a subgroup of index
$4$. In all cases we obtain a subgroup $H_0$ of index $4$ with
$\exp H_0\in\{\exp H,\,2^d\}$ and thus $8\le\exp H_0\le2^d$.
\end{proof}

\begin{proof}[Proof of Theorem~\ref{thm:main}]
Necessity was proved in Section~\ref{sec:intro}, and sufficiency for
$d\le2$ is Proposition~\ref{prop:smalld}. Let $d\ge3$.
Lemma~\ref{lem:select} provides a subgroup $H_0$ of index $4$ of $H$
with $8\le\exp H_0\le2^d$, and Theorem~\ref{thm:stock}(b) provides
a $(2m_d^2,2m_d,4,-)$ covering EBS $\mathcal E$ on
$K=\Z_3^2\times H_0$ which Gauss-covers a character. Choose $L$
with $K\le L\le\Z_3^2\times H$ and $[L:K]=2$, which is possible
since $(\Z_3^2\times H)/K$ is abelian of order $4$. By
Lemma~\ref{lem:partial} with $2m_d=u$, the group $L$ carries a
$(u^2,u,2,-)$ covering EBS which Gauss-covers a character, and
Lemma~\ref{lem:assembly2} assembles the latter into a HDS $D$ of
order $u^2$ in $\Z_3^2\times H$. The values of $D$ are roots of
unity times values of $\mathcal E$, so some character value of $D$
is a Gauss value of modulus $u$, and $D$ is non-naive by
Lemma~\ref{lem:notroot}.
\end{proof}

%%%%%%%%%%%%%%%%%%%%%%%%%%%%%%%%%%%%%%%%%%%%%%%%%%%%%%%%

\section*{Acknowledgments}
During the preparation of this work, the author used Claude (Anthropic) in order to assist with drafting and restructuring the manuscript, with checking and simplifying arguments, and with the computer verifications of the seed constructions. 
The author  takes full responsibility for the content of the publication.

\end{document}